\documentclass[10pt]{article}
\usepackage{amssymb,amsmath,amsthm}
\usepackage{fullpage}
\usepackage{graphicx}
\usepackage{color}
\usepackage{lineno}
\newtheorem{theorem}{Theorem}
\newtheorem{lemma}[theorem]{Lemma}
\newtheorem{corollary}[theorem]{Corollary}
\newtheorem{definition}[theorem]{Definition}

\newtheorem{remark}[theorem]{Remark}

\title{Achromatic number, achromatic index and diachromatic number  of circulant graphs and digraphs 
\thanks{Research supported by CONACyT-M{\' e}xico under Projects 282280, 47510664 and PAPIIT-M{\' e}xico under Project IN107218 and IN106318.}}
\author{Gabriela Araujo-Pardo\footnotemark[2] \and Juan Jos{\' e} Montellano-Ballesteros\footnotemark[2] \and Mika Olsen \footnotemark[3] \and Christian Rubio-Montiel\footnotemark[4]}

\begin{document}
\maketitle

\def\thefootnote{\fnsymbol{footnote}}
\footnotetext[2]{Instituto de Matem{\' a}ticas, Universidad Nacional Aut{\' o}noma de M{\' e}xico, Mexico City, Mexico. {\tt [garaujo|juancho]@math.unam.mx}.}
\footnotetext[3]{Departamento de Matem{\' a}ticas Aplicadas y Sistemas, UAM-Cuajimalpa, Mexico City, Mexico. {\tt
olsen@correo.cua.uam.mx}.}
\footnotetext[4]{Divisi{\' o}n de Matem{\' a}ticas e Ingenier{\' i}a, FES Acatl{\' a}n, Universidad Nacional Aut{\'o}noma de M{\' e}xico, Naucalpan, Mexico. {\tt christian.rubio@apolo.acatlan.unam.mx}.}
\begin{abstract} 
In this paper, we determine the achromatic and diachromatic numbers of some circulant graphs and digraphs each one with two lengths and give bounds for other circulant graphs and digraphs with two lengths. 
In particular, for the achromatic number we state that $\alpha(C_{16q^2+20q+7}(1,2))=8q+5$, and for the diachromatic number we state that $dac(\overrightarrow{C}_{32q^2+24q+5}(1,2))=8q+3$. In general, we give the lower bounds $\alpha(C_{4q^2+aq+1}(1,a))\geq 4q+1$ and $dac(\overrightarrow{C}_{8q^2+2(a+4)q+a+3}(1,a))\geq 4q+3$ when $a$ is a non quadratic residue of $\mathbb{Z}_{4q+1}$ for graphs and $\mathbb{Z}_{4q+3}$ for digraphs, and the equality is attained, in both cases, for $a=3$. 
Finally,
we determine the achromatic index for circulant graphs of $q^2+q+1$ vertices when the projective cyclic plane of odd order $q$ exists. 
\end{abstract}



\section{Introduction}

A \emph{complete} $k$-vertex-coloring of a graph $G$ is a vertex-coloring of $G$ using $k$ colors such that for every pair of colors there is at least two incident vertices in $G$ colored with this pair of colors. The \emph{chromatic} $\chi(G)$ and \emph{achromatic} $\alpha(G)$ numbers of $G$ are the smallest and the largest number of colors in a complete proper $k$-vertex-coloring of $G$. Therefore
\begin{equation} \label{eq1}
\chi(G)\leq\alpha(G).
\end{equation}
The concept of the achromatic number has been intensely studied in graphs since it was introduced by Harary, Hedetniemi and Prins \cite{MR0272662} in 1967. The achromatic index $\alpha_1(G)$ is defined similarly to the achromatic number $\alpha(G)$ but with edges instead of vertices; for more references of results related to these parameter see for instance \cite{MR3774452, MR2778722, MR3249588,MR3837811,MR1475893, MR2450569,MR0441778, MR1846917}.

The chromatic number and achromatic number have been generalized for digraphs by several authors \cite{MR2998438, MR3202296}. In particular, the dichromatic number, defined by V. Neumann-Lara \cite{MR693366}, and the diachromatic number, defined by the authors \cite{MR3875016}, generalize the concepts  of chromatic and achromatic numbers, respectively. An \emph{acyclic} $k$-vertex-coloring of a digraph $D$ is a $k$-vertex coloring, such that $D$ has no monochromatic cycles and a \emph{complete} $k$-vertex-coloring of a digraph $D$ is a vertex coloring using $k$ colors such that for every ordered pair $(i,j)$ of different colors, there is at least one arc $(u,v)$ such that $u$ has color $i$ and $v$ has color $j$. 
The dichromatic number $dc(D)$ and diachromatic number $dac(D)$ of $D$ are the smallest and the largest number of colors in a complete acyclic $k$-vertex-coloring of $D$. Therefore:
\begin{equation} \label{eq1D}
dc(D)\leq dac(D).
\end{equation}

Given a set $J \subseteq \{1, \ldots, n-1\}$, 
the  \emph{circulant digraph} $\overrightarrow{C}_n(J)$ is defined as the digraph with vertex set equal to $\mathbb{Z}_n$ and $A(\overrightarrow{C}_n(J))=\{ij\colon j-i\equiv s \mod n, s \in J\}$, we call $J$ the set of \emph{lengths} of $\overrightarrow{C}_n(J)$. Moreover, to obtain a  \emph{circulant graph} ${C}_n(J)$ we define the edges of ${C}_n(J)$, where in this case $J \subseteq \{1, \ldots, \lfloor\frac{n}{2}\rfloor\}$, as $E({C}_n(J))=\{ij\colon |j-i|\equiv s \mod n, s \in J\}$.

In \cite{MR3729832} it was obtained asymptotically results of the achromatic and harmonious numbers of circulant graphs. In \cite{MR0441778} it was determined that the achromatic number of the cycle $C_n$ is equal to  its achromatic index, that is, $\alpha(C_n)= \alpha_1(C_n) = \max \{k \colon k \lfloor\frac{k}{2}\rfloor \leq n \} - s(n)$, where $s(n)$ is the number of positive integer solutions to $n=2x^2+x+1$. On the other hand, in the collection of papers \cite{MR3249588,MR0384589,MR543176,MR1447777,MR989126,MR1152043,MR961664}  the achromatic index of $K_n$ was determined for $n\leq 14$, when $n=p^2+p+1$ for $p$ an odd prime power and $n=q^2+2q+1-a$ for $q$ a power of two and $a\in\{0,1,2\}$. For the no-proper version of these results see \cite{MR3774452,MR2778722,MR2080106,MR1811221}.

In this paper, we determine the achromatic and diachromatic numbers of some circulant graphs and digraphs each one with two lengths and we give bounds for other circulant graphs and digraphs with two lengths. 
In particular, for the achromatic number we state that $\alpha(C_{16q^2+20q+7}(1,2))=8q+5$, whenever $8q+5$ is a prime, and for the diachromatic number we state that $dac(\overrightarrow{C}_{32q^2+24q+5}(1,2))=8q+3$, whenever $8q+3$ is a prime. In general, we give the lower bounds $\alpha(C_{4q^2+aq+1}(1,a))\geq 4q+1$ and $dac(\overrightarrow{C}_{8q^2+2(a+4)q+a+3}(1,a))\geq 4q+3$ when $a$ is a non quadratic residue of $\mathbb{Z}_{4q+1}$ for graphs and $\mathbb{Z}_{4q+3}$ for digraphs, and the equality is attained, in both cases, for $a=3$. 
In the last section, we determine the achromatic index for circulant graphs of $q^2+q+1$ vertices when the projective cyclic plane of odd order $q$ exists.


\section{Complete colorings on circulant graphs and digraphs}

In this section we consider circulant graphs and digraphs of order prime $p$. 
We use simple upper bounds for the achromatic and diachromatic number and some properties of quadratic and non-quadratic residues.
    
An upper bound of the achromatic number of a graph $G$ with size $m$ is: 
\begin{equation} \label{eq2}
\alpha(G)\leq \left\lfloor \frac{1}{2}+\sqrt{\frac{1}{4}+2m}\right\rfloor.
\end{equation}

The authors \cite{MR3249588} determined the following upper bound of the diachromatic number of a digraph $D$ with size $m$:
\begin{equation} \label{eq3}
dac(G)\leq \left\lfloor \frac{1}{2}+\sqrt{\frac{1}{4}+m}\right\rfloor.
\end{equation}

Properties of quadratic residues, denoted by $QR$, and non quadratic residues, denoted by $NQR$, have been widely studied both cases when $p$ is an odd prime congruent with either 1 or 3 module 4. In the following remark, we list some results we use in order prove our results, see \cite{MR990017}.

\begin{remark} \label{Rem QR} 
Let $p$ be an odd prime and consider $\mathbb{Z}_p$. \\
If $p\equiv 1 \mod 4$, then
\begin{enumerate}
\item If $i\in QR$ ($i\in NQR$ resp.), then $-i\in QR$ ($-i\in NQR$ resp.).
\item Let $i\in QR$. If $j\in QR$ ($j\in NQR$ resp.) then $ij\in QR$ ($ij\in NQR$ resp.). 
\item $|QR|=\lfloor\frac{p-1}{4}\rfloor$.
\item  The integer $2\in NQR$ if and only if $p=8q+5$.
\end{enumerate}
If $p\equiv 3 \mod 4$, then
\begin{enumerate}
\item If $i\in QR$ ($i\in NQR$ resp.), then $-i\in NQR$ ($-i\in QR$ resp.).
\item $|QR|=\lfloor\frac{p-1}{2}\rfloor$.
\item  The integer $2\in NQR$ if and only if $p=8q+3$.

\end{enumerate}
\end{remark} 

\subsection {Achromatic number on circulant graphs}
Let $p$ be a an odd prime power such that $p=4q+1$ and let $QR$ be the set of quadratic residues and $NQR$ the set of non residues quadratics  in $\mathbb{Z}_{p}$. 

\begin{theorem}\label{CG}
Let $4q+1$ be an odd prime number and let $a\in NQR$. Then $\alpha(C_{4q^2+aq+1}(1,a))\geq 4q+1$.
\end{theorem}

\begin{proof}
To prove the theorem we  give a complete and proper coloring  of $C_{4q^2+aq+1}(1,a)$ with $4q+1$ colors, that is, $K_{4q+1}$ is a homeomorphic image of
the  digraph $C_{4q^2+aq+1}(1,a)$. 
Let $K_{4q+1}$ be the complete graph and consider the quadratic residue $QR_{4q+1}=\{1,r_2,\dots,r_q\}$ in $\mathbb{Z}_{4q+1}$. 

We define the following color sequences, where $2\le i\le q$:
\[
\begin{array}{lcl}
R_1 &=& (0,1,2,\dots,4q,0,1,\dots,a-3,a-2)   \\
R_2 &=& (a-1,a-1+r_2,a-1+2r_2,\dots,a-1+(4q)r_2,a-1,\\
&&a-1+r_2,a-1+2r_2,\dots,a-1+(a-2)r_2) \\
R_3 &=& (a-1+(a-1)r_2,a-1+(a-1)r_2+r_3,a-1+(a-1)r_2+2r_3,\dots,a-1+(a-1)r_2+(4q)r_3, \\ 
&& a-1+(a-1)r_2,a-1+(a-1)r_2+r_3\ldots,a-1+(a-1)r_2+(a-2)r_3) \\
R_i &=& (a-1+(a-1)r_2+(a-1)r_3\dots + (a-1)r_{i-1},\\
&&a-1+(a-1)r_2+(a-1)r_3\dots + (a-1)r_{i-1}+r_i,\\
&& a-1+(a-1)r_2+(a-1)r_3\dots + (a-1)r_{i-1}+2r_i,\ldots,\\
&& a-1+(a-1)r_2+(a-1)r_3\dots + (a-1)r_{i-1}+(4q)r_i,\\
&& a-1+(a-1)r_2+(a-1)r_3\dots + (a-1)r_{i-1},\\
&&a-1+(a-1)r_2+(a-1)r_3\dots + (a-1)r_{i-1}+r_i,\\
&& a-1+(a-1)r_2+(a-1)r_3\dots + (a-1)r_{i-1}+2r_i,\ldots,\\
&&a-1+(a-1)r_2+(a-1)r_3\dots + (a-1)r_{i-1}+(a-2)r_i)\\ 

R_q&=&(a-1+(a-1)r_2+(a-1)r_3+\dots+(a-1)r_{q-1},\\
&&a-1+(a-1)r_2+(a-1)r_3+\dots+(a-1)r_{q-1}+r_{q},\\
&& a-1+(a-1)r_2+(a-1)r_3+\dots+(a-1)r_{q-1}+2r_{q}, \dots\\
&&a-1+(a-1)r_2+(a-1)r_3+\dots+(a-1)r_{q-1}+(4q)r_{q}\\
&&a-1+(a-1)r_2+(a-1)r_3+\dots+(a-1)r_{q-1},\\
&&a-1+(a-1)r_2+(a-1)r_3+\dots+(a-1)r_{q-1}+r_{q},\\
&& a-1+(a-1)r_2+(a-1)r_3+\dots+(a-1)r_{q-1}+2r_{q}, \dots\\
&&a-1+(a-1)r_2+(a-1)r_3+\dots+(a-1)r_{q-1}+(a-2)r_{q})\\

R_{q+1} &=&(a-1+(a-1)r_2+(a-1)r_3+\dots+(a-1)r_{q})
\end{array}
\]
For $1\le i\le q$, we concatenate $R_i$ with $R_{i+1}$ and finally, we concatenate $R_{q+1} $ with $R_1$, obtaining a $4q+1$-coloring of the cycle $C_{4q^2+aq+1}$. We add the edges between vertices of distance $a$ on the cycle $C_{4q^2+aq+1}$ obtaining the graph $C_{4q^2+aq+1}(1,a)$. 
Observe that, in addition to the edges between vertices of distance $a$ in each $R_i$, 
for $1\le i\le q$, between $R_i$ and $R_{i+1}$, we 
have  the edge between the $ (4q+1)$-th  vertex of $R_i$ and the first vertex of $R_{i+1}$,
\[\{a-1+(a-1)r_2+(a-1)r_3+\dots + (a-1)r_{i-1}+(4q)r_i, a-1+(a-1)r_2+(a-1)r_3+\dots + (a-1)r_{i}\}.\]
For instance, between $R_1$ and $R_2$ we add the edges $(a-2,a-1)$, $(4q,a-1)$.
Finally, between $R_{q+1}$ and $R_1$, we have the edge between the unique vertex of $R_{q+1}$ and the $a$-th vertex of $R_1$,
\[\{a-1+(a-1)r_2+(a-1)r_3+\dots+(a-1)r_{q}, a-1\}.\]


Summarizing, the concatenation $(R_{q+1}R_1\dots R_{q}R_{q+1})$ in cyclic order contains all the edges $\{l,k\}$ for all pair $\{l,k\}\in Z_{4q+1}$, such that  in $R_i$ are the edges $\{j,j+r_i\}$ and $\{j,j+ar_i\}$ with $r_i\in QR_{4q+1}$ and $ar_i\in NQR_{4q+1}$  less the edges $\{a-1+(a-1)r_2+\dots+(a-1)r_{i-1}+4qr_i,a-1+(a-1)r_2+\dots+(a-1)r_i\}$.

For $1\le i\le q$, each sequence $R_i$ has order $4q+a$ and $R_{q+1}$ has only one element, then the order of circulant graph is equal to $4q^2+aq+1$ and we can embed $C_{4q^2+aq+1}(1,a)$ in $K_{4q+1}$, and 
\[4q+1\leq \alpha(C_{4q^2+aq+1}(1,a)).\]
\end{proof}

By Remark \ref{Rem QR},  $2\in NQR$ if and only if $p=8q+5$, applying Equation \ref{eq2} for $a=2$, we have the following result:
\begin{corollary}\label{8q+5}
Let $8q+5$ be a prime number, then $\alpha(C_{16q^2+20q+7}(1,2))=8q+5$.
\end{corollary}
Moreover, when $a=3$ we have also the following:

\begin{corollary}\label{dosytres}
Let $4q+1$ be a prime number. If $3\in NQR$, then $\alpha(C_{4q^2+aq+1}(1,3))= 4q+1$.
\end{corollary}
\begin{proof}
If $a=3$, $C_{4q^2+3q+1}(1,3)$ has $8q^2+6q+2$ edges and $K_{4q+2}$ has $8q^2+6q+1$ edges, by Equation \ref{eq2}, it follows that $\alpha(C_{4q^2+3q+1}(1,3))\leq 4q+2$. Suppose that $\alpha(C_{4q^2+3q+1}(1,3))= 4q+2$. Observe that there is only one pair of colors repeated in the coloring of the graph $C_{4q^2+3q+1}(1,3)$. Consider a homomorphism $\varphi:V(C_{4q^2+3q+1}(1,3))\rightarrow V(K_{4q+2})$. Let us consider the multigraph induced by the homomorphism $\varphi$. Since $K_{4q+2}$ is $(4q+1)$-regular and $C_{4q^2+3q+1}(1,3)$ is $4$-regular, then each vertex of the multigraph has degree at least $4q+4$. Hence, there are at least $3(4q+2)/2$ set of colors repeated (or multiedges) a contradiction. Thus, $\alpha(C_{4q^2+3q+1}(1,3))= 4q+1$.
\end{proof}

 For instance, applying Corollary \ref{8q+5}, for $q=1$ we have that $8q+5=13$, and the quadratic residues in $\mathbb{Z}_{13}$ are $QR=\{1,3,4\}$, and by the Corollary \ref{8q+5}, $\alpha(C_{43}(1,2))=13$. In this case the coloring is defined by the following sequences:

$\begin{array}{lcl}
R_1  & =  & (0,1,2,3,4,5,6,7,8,9,10,11,12,0) \\
R_2  &  = & (1,4,7,10,0,3,6,9,12,2,5,8,11,1)\\
R_3  &  = & (4,8,12,3,7,11,2,6,10,1,5,9,0,4) \\
R_4  &  = & (8).\\
\end{array}$

\subsection{Diachromatic number on circulant digraphs}

Let $p=4q+3$ be a prime. In this case $C_p(J)$ is a circulant digraph for $J$,  recall that by Remark \ref{Rem QR} if $i\in QR$, then $-i\in NQR$. 

\begin{theorem}\label{4q+3}
Let $4q+3$ be a prime number and let $a\in NQR$. Then $dac(\overrightarrow{C}_{8q^2+2(a+4)q+a+3}(1,a))\geq 4q+3$.
Moreover, for $a=3$ the equality holds.
\end{theorem}

\begin{proof}
We proceed as in the proof of Theorem \ref{CG}, considering the $2q+1$ quadratic residues and thus $2q+1$ sequences $R_1,R_2,\dots,R_{2q+1}$ each one of order $4q+a+2$ and $R_{2q+2}$ of order $1$. 

For the inequalities, when $a=3$ proceed as in Corollary \ref{dosytres} considering the complete digraph $K_{4q+4}$ with $(4q+4)(4q+3)$ arcs and $(4q+3)$ (in/out)-regular. In this case the size of $K_{4q+4}$ equals the size of $\overrightarrow{C}_{8q^2+14q+6}(1,3)$ and the (in/out)-degree of each vertex of the multidigraph is congruent with 0 module 4. Since the complete digraph $K_{4q+4}$ is $(4q+3)$ (in/out)-regular, for any pair of colors $i$ and $j$ there are $4q+3$ arcs from $i$ to $j$, but that is impossible because $\overrightarrow{C}_{8q^2+14q+6}(1,3)$ has unique edges between many pair of colors. 
\end{proof}

Moreover, applying Remark \ref{Rem QR} to Theorem \ref{4q+3} and the Equation \ref{eq3}  we have also the following result:
\begin{corollary}\label{8q+3}
Let $8q+3$ be a prime number, then $dac(\overrightarrow{C}_{32q^2+24q+5}(1,2))=8q+3$.
\end{corollary}

For instance, if $q=1$ then $8q+3=11$, the quadratic residues in $\mathbb{Z}_{11}$ are $QR=\{1,3,4,5,9\}$, and by Corollary \ref{8q+3}, $dac(\overrightarrow{C}_{61}(1,2))=11$. In this case the coloring is defined by the following sequences:

$\begin{array}{lcl}
R_1  & =  & (0,1,2,3,4,5,6,7,8,9,10,0) \\ 
R_2  &  = & (1,4,7,10,2,5,8,0,3,6,9,1)\\ 
R_3  &  = & (4,8,1,5,9,2,6,10,3,7,0,4) \\ 
R_4  &  = & (8,2,7,1,6,0,5,10,4,9,3,8)\\ 
R_5  &  = & (2,0,9,7,5,3,1,10,8,6,4,2)\\ 
R_6  &  = & (0).\\
\end{array}$


\section{Achromatic index of circulant graphs}

In this section, we endeavor to determine the achromatic index of some circulant graphs. In order to obtain upper bounds, we analyze the behavior of some functions. While to obtain lower bounds, we exhibit a proper and complete edge-coloring of $C_n(J)$ using some definitions and remarks about projective planes.

A \emph{projective plane} consists of a set of $n$ points, a set of lines, and an incidence relation between points and lines having the following properties:
\begin{itemize}
\item Given any two different points there is exactly one line incident to both of them.
\item Given any two different lines there is exactly one point incident to both of them.
\item There are four points, such that no line is incident to more than two of them.
\end{itemize}
For some number $q$, such a plane has $n=q^2+q+1$ points and $n$ lines; each line contains $q+1$ points and each point belongs to $q+1$ lines. The number $q$ is called the \emph{order} of the projective plane. If a projective plane of order $q$ exists then it is denoted by $\Pi_q$. When $q$ is a prime power, a projective plane of order $q$ does exist which is denoted by $PG(2,q)$ owing to the fact that it arising from the finite field of order $q$ and it is called \emph{the algebraic projective plane}.

Let $\mathbb{P}$ be the set of points of $\Pi_q$ and let $\mathbb{L}= \{L_0,\dots, L_{n-1}\}$ be the set of lines of $\Pi_q$. The complete graph $K_n$ can be seen to have the vertex set $\mathbb{P}$, and, for $i\in \{0,\ldots,n-1\}$, the line $L_i$ can be interpreted as the subgraph $K_{q+1}$ of $K_n$ induced by the points of $L_i$; we denote this graph by $l_i$. From the properties of $\Pi_q$ it follows that if $i$, $j \in \{0,\dots,n-1\}$, $i\not=j$, then $|V(l_i)\cap V(l_j)| =1$. Moreover, $\{E(l_0), \dots, E(l_{n-1})\}$ is a partition of $E(K_n)$. 

Now, we recall the concept of difference sets of a $\mathbb{Z}_n$ group. A subset $D=\{d_0,d_1,\dots,d_q\}$ of $\mathbb{Z}_n$ is a \emph{difference set} if for each $g\in \mathbb{Z}_n$, $g\not = 0$, there is a unique pair of different elements $d_i,d_j\in D$ such that $g=d_i-d_j$. Therefore, if $D\subset \mathbb{Z}_n$ is a difference set and $i\in \mathbb{Z}_n$ is an arbitrary element, then both $-D:=\{-d_0,-d_1,\dots,-d_q\}$ and $D+i:=\{d_0+i,d_1+i,\dots,d_q+i\}$ are difference sets. See Table \ref{tab1} for some examples.

\begin{table}[!htbp]
\begin{center}
\begin{tabular}{|c|c|c|c|}
\hline
$\mathbb{Z}_n$ & Difference set & $\mathbb{Z}_n$ & Difference set\\
\hline
$\mathbb{Z}_{13}$ & 1,2,5,7 & $\mathbb{Z}_{91}$ & 1,2,4,10,28,50,57,62,78,82\\
\hline
$\mathbb{Z}_{31}$ & 1,2,4,9,13,19 & $\mathbb{Z}_{133}$ & 1,2,4,13,21,35,39,82,89,95,105,110 \\
\hline
$\mathbb{Z}_{57}$ & 1,2,4,14,33,37,44,53 & $\mathbb{Z}_{183}$ & 1,2,4,17,24,29,43,77,83,87,120,138,155,176\\
\hline
\end{tabular}
\caption{\label{tab1}Some known different sets for some $\mathbb{Z}_n$.}
\end{center}
\end{table}

Next, we recall the notion of cyclic planes using the polygon model. Let $q>1$ be an integer, and $n=q^2+q+1$. If the group $ \mathbb{Z}_n$ contains a difference set $D=\{d_0,d_1,\dots,d_q\}$ then there exists a projective plane called \emph{cyclic projective plane of order $q$}, defined as follows: the points are the elements of $\mathbb{Z}_n$ that is the set of integers $\{0,1,\dots,n-1\}$, and the lines are the sets $\{l_i\}_{i=0}^{n-1}$, where $l_i:=D+i$. Throughout the paper when we deal with elements of $\mathbb{Z}_n$ all sums are taken modulo $n$.

The class of cyclic projective planes is wider than the class of algebraic projective planes, but each known finite cyclic plane is isomorphic to $PG(2,q)$ for a suitable $q$. The following representation of the cyclic planes comes from K{\' a}rtesi \cite{MR0423175}, and it is useful to illustrate our proofs: Consider the numbering of the vertices of a regular $n$-gon with the elements of $\mathbb{Z}_n$ in clockwise order. Note that the subpolygon with $q+1$ vertices induced by a difference set $D$ has the property that all the chords obtained by joining pairs of their points have different lengths, and it represents the line $l_0$ of $\Pi_q$. Moreover the line $l_i$ is obtained by rotating $l_0$ around the center by angle $2\pi \frac{i}{n}$ for $i\in\{1,\dots,n-1\}$. See Figure \ref{Fig1}.

\begin{figure}[ht!]
\begin{center}
\includegraphics{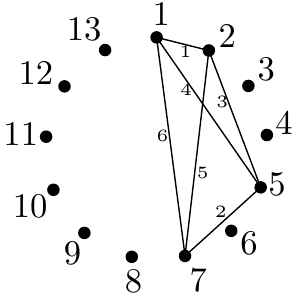}
\caption{\label{Fig1} The line $l_0$ of the cyclic model of $\Pi_{3}$ as a polygon of $13$ vertices, i.e., the line $l_0$ of the circulant graph $C_{13}(1,2,3,4,5,6)$.}
\end{center}
\end{figure}

Finally, we recall that any complete graph of even order $q+1$ admits a $1$-factorization with $q$ perfect matchings and we define the following definitions in order to simplify the proofs. 

\begin{definition} 
Let $q+1$ be an even integer, $\{F_1,\dots,F_{t}\}$ a set of perfect matching pairwise edge-disjoint of $K_{q+1}$ for some $t\in\{1,\dots,q\}$ and $W=\bigcup\limits_{j=1}^{t}F_{j}$. An edge-coloring $\varsigma\colon E(W) \rightarrow \mathcal{C}$, for $\mathcal{C}=\{c^1,\dots,c^t\}$ a set of $t$ colors, will say to be of \emph{Type $M_t$} if for every $j\in \{1,2,\dots,t\}$ the set $\{ xy\in E(W) \colon \varsigma(xy) = c^j\}$ is $F_j$.

Since each color class is a matching, we will say that $W$ is an \emph{owner} of the set of colors $\mathcal{C}$.
\end{definition}

In order to prove our theorem, we show the following lemma.

\begin{lemma}\label{lema2} 
Let  $n= q^2 + q+ 1$ with $q$ a natural number, such that $ \Pi_q $ exists. Let $K_n$ be a representation of $\Pi_q$ and let $\varsigma\colon E(K_{n})\rightarrow \mathcal{C}$ be a partial edge-coloring of $K_{n}$. Suppose that each line $l_i$  of $K_n$ is an owner of the set of colors $\mathcal{C}_i\subseteq \mathcal{C}$. Then $\varsigma$ is complete.
\end{lemma}
\begin{proof} 
Let $\{c_1, c_2\}\subseteq \bigcup\limits_{i=1}^{n} \mathcal{C}_i$. If there is an $i\in \{1,\dots,n\}$, such that $\{c_1, c_2\}\subseteq \mathcal{C}_i$, since $l_i$ is an owner of $\mathcal{C}_i$ it follows that each $u\in V(l_i)$ is incident with edges colored $c_1$ and $c_2$. 

If $c_1\in \mathcal{C}_i$ and $c_2\in\ \mathcal{C}_j$ with $i\not=j$ there is $u\in V(G)$, such that $u= V(l_i)\cap V(l_j)$ and since $l_i$ and $l_j$ are owners of $\mathcal{C}_i$ and $\mathcal{C}_j$, respectively, by the properties of the projective plane, $u$ is incident with edges colored $c_1$ and $c_2$.
\end{proof}

\begin{theorem}\label{teo9}
Let $\Pi_q$ be a cyclic projective plane of order odd $q$ with difference set $D$, $n=q^2+q+1$ and $C_n(J)$ a circulant graph such that $J\subseteq \{1,\dots,\left\lfloor \frac{n}{2}\right\rfloor\}$. If the chords of $D$ have lengths of $J$ of $l_0$ and they are the union of $t$ matchings ($1\leq t\leq q$), then: \[tn\leq\alpha_1(C_n(J)).\]
\end{theorem}
\begin{proof}
In order to prove this theorem, we exhibit a proper and complete edge-coloring of $C_{n}(J)$ for $n=q^2+q+1$, with $k:=tn$ colors. Since $J$ is a subset of $\{1,\dots,\left\lfloor \frac{n}{2}\right\rfloor\}$ and $\left\lfloor \frac{n}{2}\right\rfloor = \binom{q+1}{2}$, then $t=2|J|/(q+1)$. Let $\mathcal{C}$ be a set of $k$ colors, and let $\{\mathcal{C}_0, \dots, \mathcal{C}_{n-1}\}$ be a partition of $\mathcal{C}$ such that $\mathcal{C}_i$ is a set of $t$ colors for all $i\in\{0,\dots,n-1\}$. 

Let $V(C_n(J))$ be representing the points of a cyclic plane $\Pi_q$, and let $\{l_0,\dots,l_{n-1}\}$ be the set of lines of induced in $V(C_n(J))$ such that $V(l_0)=D$.

To color the edges of $V(C_n(J))$ with $\varsigma\colon V(C_n(J)) \rightarrow [k]$, we assign partial colorings of Type $M_i$ using the colors of $\mathcal{C}_i$ to the lines of $l_i$ for all $i\in\{0,\dots,n-1\}$.

Since each color class is a matching, the coloring is proper. Since each line $l_i$ is an owner of each color of $\mathcal{C}_i$ ($1\leq i\leq n$) by Lemma \ref{lema2} it follows that the resultant edge-coloring $ \varsigma:=\overset{n}{\underset{i=1}{\bigcup}}\varsigma_{i} $ of $V(C_n(J))$ is a complete edge-coloring using $k$ colors.
\end{proof}

Next, we analyze the upper bounds. Note that a circulant graph $G=C_n(J)$ is $r$-regular with $r=2|J|$, therefore, it has size $n\binom{r}{2}$ and by Equation \ref{eq2}, we have that: \[\alpha_1(G)\leq \left\lfloor \frac{1}{2}+\sqrt{\frac{1}{4}+nr(r-1)}\right\rfloor.\]

This equation works better with some graphs (normally with sparse graphs, i.e., those for which its size is much less than $|V|^2$) but for these graphs we give a different approach to obtain an upper bound, firstly introduced for complete graphs, see \cite{MR989126}. Let $f_{n,r}(x)$ and $g_{n,r}(x)$ be functions that counts the maximum number of chromatic classes in two different ways where $x$ is the size of the smallest chromatic class defined as follows.
\[f_{n,r}(x)= \frac{nr}{2x}\textrm{ and }g_{n,r}(x)=\begin{cases}
\begin{array}{c}
2x(r-1)+1\\
x(n-2x+r-1)+1
\end{array} & \begin{array}{c}
\textrm{ if }r< n-2x;\\
\textrm{ if }r\geq n-2x.
\end{array}\end{cases}\]

Since $f_{n,r}$ is a decreasing function and $g_{n,r}$ is increasing function, we get that: 
\[\alpha_1(G)\leq\mathrm{max}\left\{ \mathrm{min}\{\left\lfloor f_{n,r}(x)\right\rfloor,g_{n,r}(x) \colon x\in\mathbb{N}\}\right\},\]
and thus,
\[\alpha_1(G)\leq\mathrm{max}\left\{ \mathrm{min}\{f_{n,r}(x),g_{n,r}(x) \colon x\in\mathbb{R}\}\right\}.\]

\begin{theorem}\label{teo10}
Let $n=q^2+q+1$ be and $C_n(J)$ a circulant $r$-regular graph such that $r=(q+1)t$ with ($1\leq t\leq q$), and there exists $\Pi_q$ a cyclic projective plane of order odd $q$ such that $J\subseteq \{1,\dots,\left\lfloor \frac{n}{2}\right\rfloor\}$, the chords of $D$ have lengths of $J$ of $l_0$ and they are the union of $t$ matchings, then: \[\alpha_1(C_n(J))=tn.\]
\end{theorem}
\begin{proof}
Since the result is true for complete graphs, assume ($1\leq t\leq q-1$).  By Theorem \ref{teo9}, $tn\leq\alpha_1(C_n(J))$ and $\alpha_1(C_n(J))\leq f_{n,r}(x)= \frac{n(q+1)t}{2x}$ and then $x\leq\frac{q+1}{2}$.

Suppose $x=\frac{q+1}{2}$, then $f_{n,r}(x)=tn$ and $g_{n,r}(x)=((q+1)t-1)(q+1)+1=tn+tq-q=$ since $(q+1)\leq t(q+1)\leq q^2-1<q^2=n-2x$. Hence $\mathrm{min}\{\left\lfloor f_{n,r}(x)\right\rfloor,g_{n,r}(x) \}=f_{n,r}(x)=tn=q^2t+qt+t$.

 Now, suppose $x=\frac{q-1}{2}$, then $f_{n,r}(x)=tn\frac{q+1}{q-1}=\frac{q^3t+2q^2t+2qt+t}{q-1}$ and $g_{n,r}(x)=((q+1)t-1)(q-1)+1=\frac{q^3t-q^2t-qt+t-(q^2-3q+2)}{q-1}$ (newly, since $(q+1)\leq t(q+1)\leq q^2-1<q^2=n-2x$). Hence $\mathrm{min}\{\left\lfloor f_{n,r}(x)\right\rfloor,g_{n,r}(x) \}=g_{n,r}(x)=q^2t-t-q+2$ since $q\geq 3$.
 
Therefore, \[\alpha_1(C_n(J))\leq\mathrm{max}\left\{q^2t+qt+t,q^2t-t-q+2\}\right\}=q^2t+qt+t=tn,\]
i.e., \[\alpha_1(C_n(J))=tn.\]

\end{proof}

For example, $\alpha_1(C_n(1,2))=\alpha_1(C_n(3,6))=\alpha_1(C_n(4,5))=13$, $\alpha_1(C_n(3,4,5,6))=\alpha_1(C_n(1,2,4,5))=\alpha_1(C_n(1,2,3,6))=26$ and $\alpha_1(C_n(1,2,3,4,5,6))=39$.


\bibliographystyle{plain}
\bibliography{biblio}

\end{document}